\documentclass[reqno]{amsart}
\usepackage{amsmath}
\usepackage{amscd,amsthm,amsfonts,amsopn,amssymb}
\usepackage[bookmarksopen]{hyperref}
\usepackage{epsfig}
\usepackage{graphicx}
\numberwithin{equation}{section}

\makeatletter
\def\subsection{\@startsection{subsection}{2}
  \z@{0.0\linespacing}{-.5em}
  {\normalfont\bfseries}}
\makeatother

\newtheorem{theorem}{Theorem}[section]
\newtheorem{proposition}[theorem]{Proposition}

\newtheorem{lemma}[theorem]{Lemma}
\theoremstyle{remark}
\newtheorem{remark}[theorem]{Remark}
\DeclareMathOperator{\supp}{supp\,}
\DeclareMathOperator{\divop}{div}
\DeclareMathOperator{\Tr}{Tr\,}
\DeclareMathOperator{\tr}{tr\,}
\DeclareMathOperator{\rad}{rad}

\title[Negative energy blowup for the Hartree hierarchy]{Negative energy blowup results for the focusing Hartree hierarchy via identities of virial and localized virial type}
\author{Aynur Bulut}
\address{Department of Mathematics, Louisiana State University}
\email{aynurbulut@lsu.edu}
\allowdisplaybreaks

\begin{document}
\begin{abstract}
We establish virial and localized virial identities for solutions to the Hartree hierarchy, an infinite system of partial differential equations which arises in mathematical modeling of many body quantum systems.  As an application, we use arguments originally developed in the study of the nonlinear Schr\"odinger equation (see work of Zakharov, Glassey, and Ogawa--Tsutsumi) to show that certain classes of negative energy solutions must blow up in finite time.  

The most delicate case of this analysis is the proof of negative energy blowup without the assumption of finite variance; in this case, we make use of the localized virial estimates, combined with the quantum de Finetti theorem of Hudson and Moody and several algebraic identities adapted to our particular setting.  Application of a carefully chosen truncation lemma then allows for the additional terms produced in the localization argument to be controlled.
\end{abstract}
\thanks{January 2018}
\maketitle
\tableofcontents
\section{Introduction}

Fix $d\geq 2$, and let $V\in\mathcal{S}(\mathbb{R}^d;\mathbb{R})$ be a smooth bounded potential in the Schwartz space with even symmetry.  In this paper, we study an infinite system of coupled PDEs, often referred to as the {\it Hartree hierarchy}, which arises in the study of many-body quantum mechanics.  This hierarchy describes the evolution of a sequence $(\gamma^{(k)})_{k\geq 1}$ with each $\gamma^{(k)}$ mapping $\mathbb{R}^{dk}\times\mathbb{R}^{dk}$ into $\mathbb{C}$ satisfying symmetry properties matching those of the {\it factorized} profiles 
\begin{align}
\gamma^{(k)}(x_1,x_2,\cdots,x_k,x'_1,x'_2,\cdots,x'_k)=\phi(x_1)\cdots\phi(x_k)\overline{\phi(x'_1)\cdots\phi(x'_k)}\label{factorized}
\end{align}
for $x_1,\cdots,x_k,x'_1,\cdots,x'_k\in\mathbb{R}^d$.  We begin by specifying our conventions for differentiation and the Fourier transform, which are motivated by the form (\ref{factorized}).  In particular, for $k\geq 1$ and $\gamma^{(k)}\in L^2(\mathbb{R}^{dk}\times \mathbb{R}^{dk})$, we take the Fourier transform in $x_i$ and $x_i'$, $i=1,\cdots, k$ to be defined by, respectively
\begin{align*}
&\mathcal{F}_{x_i}[\gamma^{(k)}](x_1,\cdots,x_{i-1},\xi_i,x_{i+1},\cdots, x_k,x'_1,\cdots,x'_k)\\
&\hspace{0.6in}=\int e^{-ix_i\xi_i}\gamma^{(k)}(x_1,\cdots,x_k,x'_1,\cdots,x'_k)dx_i,\\
&\mathcal{F}_{x'_i}[\gamma^{(k)}](x_1,\cdots,x_k,x'_1,\cdots,x'_{i-1},\xi'_i,x'_{i+1},\cdots, x'_k)\\
&\hspace{0.6in}=\int e^{ix'_i\xi'_i}\gamma^{(k)}(x_1,\cdots,x_k,x'_1,\cdots,x'_k)dx'_i.
\end{align*}
With this notation, we have the standard identities $\nabla_{x_i} \gamma^{(k)}=[(i\xi_i)\widehat{\gamma^{(k)}}]^\vee$ and 
$\nabla_{x'_i} \gamma^{(k)}=[(-i\xi'_i)\widehat{\gamma^{(k)}}]^\vee$ for $1\leq i\leq k$.

Correspondingly, we let $\Delta^{(k)}_\pm$ be the operator given by
\begin{align*}
\Delta^{(k)}_\pm:=\sum_{i=1}^k\Delta_{x_i}-\Delta_{x_i'}.
\end{align*}

The Hartree hierarchy is then given by
\begin{align}
\label{label_1}\left\lbrace\begin{array}{c}i\partial_t \gamma^{(k)}+\Delta^{(k)}_\pm\gamma^{(k)}=\mu\sum_{j=1}^k B_{j,k+1}^\pm\gamma^{(k+1)},\\
\gamma^{(k)}(0,{\mathbf x},{\mathbf x}')=\gamma^{(k)}_0({\mathbf x},{\mathbf x}'),\end{array}\right.
\end{align}
for functions $(t,{\mathbf x},{\mathbf x}')\mapsto \gamma^{(k)}(t,{\mathbf x},{\mathbf x}')$, $k\geq 1$, where the variables $t$, ${\mathbf x}$, and ${\mathbf x}'$ belong, respectively, to a time interval $I\subset \mathbb{R}$ and the spaces $\mathbb{R}^{dk}$ and $\mathbb{R}^{dk}$.  Here, we take $\mu\in \{-1,1\}$, and 
\begin{align*}
B_{j,k+1}^\pm\gamma^{(k+1)}=B_{j,k+1}^+\gamma^{(k+1)}-B_{j,k+1}^-\gamma^{(k+1)},
\end{align*}
with $B_{j,k+1}^\pm$ defined by
\begin{align*}
(B_{j,k+1}^+\gamma^{(k+1)})(t,{\mathbf  x},{\mathbf x}')&:=\int \gamma^{(k+1)}({\mathbf x},y,{\mathbf x}',y)V(x_j-y)dy,\\
(B_{j,k+1}^-\gamma^{(k+1)})(t,{\mathbf x},{\mathbf x}')&:=\int \gamma^{(k+1)}({\mathbf x},y,{\mathbf x}',y)V(x_j'-y)dy,
\end{align*}
under the notational conventions ${\mathbf x}=(x_1,\cdots,x_k)$ and ${\mathbf x}'=(x'_1,\cdots,x'_k)$, where $x_i$ and $x'_i$ are vectors in $\mathbb{R}^d$ for $1\leq i\leq k$.  

This hierarchy arises in the study of the mean field limit of quantum mechanical systems, where the pairwise interactions between quantum particles are governed by the potential $V$ (see for instance \cite{S}).  When $V$ is formally taken to be the Dirac measure $\delta(x)$, $(\ref{label_1})$ becomes the {\it Gross-Pitaevskii (GP) hierarchy} (in fact, derivations along these lines can be made precise by considering a sequence of approximating potentials $V_N$, scaled appropriately as $N\rightarrow\infty$; see, e.g. \cite{ESY,ESY2,KM} as well as \cite{CPT,CPHigher,CH,CHPS,CHPS2,KSS,S} and the references cited therein).  

The equation ($\ref{label_1}$) enjoys a special relationship with the classical Hartree equation (that is, the nonlinear Schr\"odinger equation with nonlocal convolution-type nonlinearity).  In particular, when the initial data $(\gamma_0^{(k)})$ has factorized form 
\begin{align*}
\gamma_{0}^{(k)}({\mathbf x},{\mathbf x}')=\prod_{j=1}^k \phi_0(x_j)\overline{\phi_0(x_j')},\quad k\geq 1,
\end{align*}
with $\phi_0\in H^2(\mathbb{R}^d)$, the function
\begin{align} 
\gamma^{(k)}(t,{\mathbf x},{\mathbf x}')&:=\prod_{j=1}^k \phi(t,x_j)\overline{\phi(t,x_j')}\label{label_2}
\end{align}
is a particular solution of ($\ref{label_1}$) provided that $t\mapsto \phi(t)$ is a solution of the Hatree equation
\begin{align}
\label{label_3}i\partial_t\phi+\Delta \phi=\mu (V*|\phi|^2)\phi
\end{align}
with $\phi(0)=\phi_0$.  In accordance with the usual nomenclature for ($\ref{label_3}$) (and, more generally, nonlinear Schr\"odinger equations), we will say that the hierarchy ($\ref{label_1}$) has {\it defocusing} nonlinearity when $\mu=1$, and {\it focusing} nonlinearity when $\mu=-1$.  Similarly, the GP hierarchy has an analogous relationship with the cubic nonlinear Schr\"odinger equation.

\vspace{0.2in}

The present work is motivated by work of Chen, Pavlovic, and Tzirakis \cite{CPT}, in which the authors establish mass conservation, energy conservation, and virial identities for the GP hierarchy, and use these along to establish negative energy blowup for finite variance initial data in that setting.  Our analysis also relies on a class of ideas surrounding the quantum de Finetti theorem of Hudson and Moody \cite{HM}, which has recently become a key tool in the analysis of hierarchies of the form ($\ref{label_1}$) -- for an overview of such arguments, we refer to work of Lewin, Nam and Rougerie \cite{LNR} in the setting of the Hartree hierarchy, and works of Chen, Hainzl, Pavlovic and Seringer \cite{CHPS,CHPS2} for the GP hierarchy (where the quantum de Finetti result played a key role in the proof of unconditional uniqueness for the hierarchy), as well as the references cited in these works.  

\vspace{0.2in}

In this paper, we continue this line of study to establish a class of identities for solutions to ($\ref{label_1}$) which hold in analogy to the usual virial identities for nonlinear Schr\"odinger equations.  As an application of these ideas, we invoke the classical argument of Glassey \cite{G} to show that certain classes of negative energy solutions must blow up in finite time, under the assumption of initially finite variance (see Theorem $\ref{label_5}$ below).  Moreover, under certain hypotheses on the potential $V$, we establish negative energy blowup in the absence of the finite variance assumption.  This result is in the spirit of a result due to Ogawa and Tsutsumi \cite{OT} for the nonlinear Schr\"odinger equation; see also work of Hirata \cite{H} for a related result for the Hartree equation.  The main tool (and the main novelty of the present work) is the derivation of a class of localized virial identities, the derivation of which in the hierarchy setting produces a number of additional terms which must be dealt with carefully.

\vspace{0.2in}

We now prepare some notation to state our main results.  Throughout the paper, we will make frequent use of the trace operator, given for $f:\mathbb{R}^{dk}\times\mathbb{R}^{dk}\rightarrow \mathbb{C}$ by
\begin{align*}
\Tr(f)&:=\int f({\mathbf x},{\mathbf x})d{\mathbf x}.
\end{align*}
with ${\mathbf x}=(x_1,\cdots,x_k)\in\mathbb{R}^{dk}$ as before, as well as the partial trace in the last variable, 
\begin{align*}
\Tr_k(f)(x_1,\cdots,x_{k-1},x'_1,\cdots,x'_{k-1})&:=\int f({\mathbf x},{\mathbf x})dx_k
\end{align*}
for $k\geq 1$, and use the notational convention
\begin{align*}
\gamma^{(k)}(t,{\mathbf x},{\mathbf x}')=\gamma^{(k)}({\mathbf x},{\mathbf x}'),\quad t\in I,
\end{align*}
omitting explicit specification of the time variable when there is no potential for confusion.  

\vspace{0.2in}

We say that a solution $(\gamma^{(k)})_{k\geq 1}$ to ($\ref{label_1}$) (in the sense of the integral Duhamel formulation) is ``(A)--(D)-admissible'' if it satisfies the following properties for all $k\in\mathbb{N}$:
\begin{enumerate}
\item[(A)] $\gamma^{(k)}\in C(I;H^2(\mathbb{R}^{dk}\times \mathbb{R}^{dk}))$, $\Tr \gamma^{(k)}=1$, and $\gamma^{(k)}\succeq 0$, in the sense that 
\begin{align*}
\bigg\langle \psi,\int_{\mathbb{R}^d}\gamma^{(k)}({\mathbf x},\cdot)\psi(\mathbf{x})d{\mathbf x}\bigg\rangle_{L^2_{{\mathbf x}'}(\mathbb{R}^{dk})}\geq 0\quad \textrm{for every}\quad \psi\in L^2,
\end{align*}
\item[(B)] $\gamma^{(k)}$ is symmetric with respect to permutations of the variables ${\mathbf x}$ and permutations of the variables ${\mathbf x}'$: for every $\sigma\in S_k$ let $P_\sigma$ denote the map $(x_1,x_2,\cdots,x_k)\mapsto (x_{\sigma(1)},x_{\sigma(2)},\cdots,x_{\sigma(k)})$.  Then for every $\sigma,\tau\in S_k$, one has
\begin{align*}
\gamma^{(k)}(t,{\mathbf x},{\mathbf x}')=\gamma^{(k)}(t,P_\sigma(\mathbf{x}),P_\tau(\mathbf{x}')).
\end{align*}
\item[(C)] $\gamma^{(k)}$ is Hermitian:
\begin{align*}
\gamma^{(k)}(t,{\mathbf x},{\mathbf x}')=\overline{\gamma^{(k)}(t,{\mathbf x}',{\mathbf x})}, 
\end{align*}
and,
\item[(D)] $\gamma^{(k)}$ is admissible:
\begin{align*}
\gamma^{(k)}(t)=\Tr_{k+1}(\gamma^{(k+1)}(t)),\,\textrm{that is,}
\end{align*}
\begin{align*}\gamma^{(k)}(t,{\mathbf x},{\mathbf x}')=\int \gamma^{(k+1)}(t,{\mathbf x},x_{k+1},{\mathbf x}',x_{k+1})dx_{k+1}.
\end{align*}
\end{enumerate}

As we note in Section $2$ below, the properties (A)--(D) are preserved under the evolution---that is, if the sequence of initial data $(\gamma^{(k)}_0)_{k\geq 1}$ satisfies (A)--(D) as functions of ${\mathbf x}$ and ${\mathbf x'}$, then the corresponding solution is (A)--(D)-admissible.  For treatment of this invariance in the case of the Gross--Pitaevskii hierarchy, see Section $5$ and Appendix B in \cite{ChTa}; we remark that the arguments for ($\ref{label_1}$) are similar.

\vspace{0.2in}

As we will see below, (A)--(D)-admissible solutions to (\ref{label_1}) obey conservation of two relevant quantities: the mass $\Tr(\gamma_0^{(1)}(t))$, and the energy
\begin{align}
E(t):=-\frac{1}{2}\Tr(\Delta_{x_1}\gamma^{(1)}(t))+\frac{\mu}{4}\Tr(B_{1,2}^+\gamma^{(2)}(t)).\label{label_4}
\end{align}

These conserved quantities play a fundamental role in the analysis of long-time and global properties of the evolution (see also \cite{CPT,CPHigher} for related results concerning the GP hierarchy).  Indeed, when the sign $\mu$ of the nonlinearity is positive, the conserved energy $E(t)$ gives uniform-in-time control over each of its component terms; in a variety of settings, this information is sufficient to conclude that solutions exist globally in time.  On the other hand, no such control is guaranteed when $\mu$ is negative, and, as shown in the work of Zakharov \cite{Z} and Glassey \cite{G}, an initial negative value for the energy leads to finite-time blowup results for the nonlinear Schr\"odinger equation under an assumption of initially finite variance (see also \cite{C} for a comprehensive treatment of these and related results).  

\vspace{0.2in}

Our first theorem is a variant of Glassey's argument, adapted to the hierarchy ($\ref{label_1}$).  To state this result, we will make use of the quantity
\begin{align*}
V_1(t)=\Tr(|x|^2\gamma^{(1)}(t)).
\end{align*}
for solutions $(\gamma^{(k)})_{k\geq 1}$ defined on a time interval $I\subset\mathbb{R}$, and $t\in I$.

\begin{theorem}
\label{label_5}
Fix $d\geq 2$, $\mu=-1$ and suppose that $V\in\mathcal{S}(\mathbb{R}^d;\mathbb{R})$ is a bounded even function such that
\begin{align*}
V+\frac{1}{2}x\cdot \nabla V\leq 0.
\end{align*}

Let $(\gamma^{(k)})_{k\geq 1}$ be an (A)--(D)-admissible solution to (\ref{label_1}) defined on an interval $I\subset\mathbb{R}$ with
\begin{align*}
V_1(0)<\infty
\quad \textrm{and}\quad E(0)<0.
\end{align*}
Then $I$ is bounded.
\end{theorem}

As we remarked above, a similar result was established for the GP hierarchy in \cite{CPT}; see also another instance of a related argument in \cite{CHPS2}.  In view of Theorem $\ref{label_5}$, it is natural to consider whether the finite variance condition $V_1(0)<\infty$ can be relaxed.  For the nonlinear Schr\"odinger equation, a partial result in this direction has been given by Ogawa and Tsutsumi \cite{OT} (see also work of Hirata \cite{H} for a related result concerning certain instances of the Hartree equation).  In our setting, we prove the following theorem:
\begin{theorem}
\label{label_6}
Fix $\mu=-1$ and suppose that $V\in\mathcal{S}(\mathbb{R}^d;\mathbb{R})$ is a bounded even function with $V\geq 0$, 
\begin{align}
V+\frac{1}{2}x\cdot \nabla V\leq 0,\label{label_7}
\end{align}
\begin{align}
\label{label_8}\sup_{|x|\geq R}|x|\,|(\nabla V)(x)|\rightarrow 0
\end{align}
as $R\rightarrow\infty$, and
\begin{align}
\label{label_9} \frac{\lVert\, |x|\, |\nabla V(x)|\, \rVert_{L^1(|x|\leq R^{1/2})}}{R^{(d-1)/2}}\rightarrow 0
\end{align}
as $R\rightarrow\infty$.

Then, if $(\gamma^{(k)})_{k\geq 1}$ is any (A)--(D)-admissible solution to (\ref{label_1}) defined on an interval $I\subset\mathbb{R}$ with $\gamma_0^{(1)}(x,x')$ radial in $x$ and $x'$ (in the sense that $\gamma_0^{(1)}(x,x')=\gamma_0^{(1)}(|x|,|x'|)$) and $\gamma_0^{(2)}(x,y,x',y')$ radial in each of $x$, $x'$, $y$ and $y'$, then the condition 
\begin{align*}
E(0)<0
\end{align*}
implies $I$ is bounded, where $E(t)$ is as defined in (\ref{label_4}).
\end{theorem}

As we briefly described above, the proof of Theorem $\ref{label_6}$ is based on a localized form of the virial identities used to prove Theorem $\ref{label_5}$.  We establish these identities in Proposition $\ref{label_34}$, with some additional algebraic tools useful in our arguments in the subsequent Lemma $\ref{label_40}$ and Lemma $\ref{label_43}$.  

The analysis leading to the proof of this result consists in the application of essentially three main ingredients:
\begin{enumerate}
\item the formulation of an appropriate form of a radial truncation lemma, Lemma $\ref{label_18}$; this truncation lemma allows for the additional terms appearing in the localized identities to be controlled by the conserved energy,
\item the quantum de Finetti theorem of \cite{HM} (see \cite{LNR}, \cite{CHPS,CHPS2}, and the references cited therein, for earlier applications of this theorem in the hierarchy context); this provides the algebraic structure necessary to estimate the ``decoupled'' nonlinearity via standard convolution estimates, and
\item decay properties of the functions $\gamma^{(k)}$ arising from the Strauss lemma, ensured by our radiality assumption.
\end{enumerate}

\hspace{0.2in}

We conclude this introduction with a brief outline of the rest of the paper.  In Sections $2$ and $3$, we establish some further notational conventions, establish conservation of mass and energy for ($\ref{label_1}$) (note that similar conservation laws were obtained by Chen, Pavlovic and Tzirakis in \cite{CPT} for the related Gross-Pitaevskii hierarchy).  Section $4$ is then devoted to the derivation of the relevant virial and localized virial identities which form the basis of our subsequent analysis.  The proofs of Theorem $\ref{label_5}$ and Theorem $\ref{label_6}$ are then given in Section $5$.  Some auxiliary technical results are established in the appendices.

\vspace{0.2in}

\subsection*{Acknowledgements}

The author would like to thank T. Chen and N. Pavlovic for valuable conversations concerning the Gross-Pitaevskii and Hartree hierarchies.  The author was partially supported by NSF grants DMS-1361838 and DMS-1748083 during the preparation of this work.

\section{Preliminaries}

In the rest of this paper, we let $V\in\mathcal{S}(\mathbb{R}^d;\mathbb{R})$ be a smooth bounded Schwartz-class potential with even symmetry.  We begin by collecting some preliminary results concerning ($\ref{label_1}$).   The first of these is a result expressing local well-posedness of the evolution in an appropriate function space.  Following \cite{CP_Cauchy,CPHigher} (see also the references cited in these works for related background), the relevant function spaces are parametrized both by a regularity parameter (often denoted $\alpha$, which we will take at the level $\alpha=1$ in the discussion below) and a scaling parameter $\xi$.  The interplay between these parameters plays a key role in the analysis.

For $(\gamma^{(k)})_{k\geq 1}$ such that $\gamma^{(k)}:\mathbb{R}^{dk}\times\mathbb{R}^{dk}\rightarrow \mathbb{C}$ for $k\geq 1$, and $\xi\in (0,\infty)$, define
\begin{align*}
\lVert (\gamma^{(k)})\rVert_{H^1_{\xi}}&:=\sum_{k\geq 1} \xi^k\left(\Tr(|\langle\nabla_x\rangle\langle\nabla_{x'}\rangle\gamma^{(k)}|^2)\right)^{1/2}.
\end{align*}

\begin{proposition}
\label{label_prop1}
Suppose that $(\gamma_0^{(k)})_{k\geq 1}$ is a sequence of functions with $\gamma_0^{(k)}:\mathbb{R}^{dk}\times\mathbb{R}^{dk}\rightarrow\mathbb{C}$ for $k\geq 1$ which satisfies conditions (A)--(D).  If $(\gamma_0^{(k)})\in H^1_{\xi}$ for some $\xi>0$, then for every $\xi'$ sufficiently small there exists $T>0$ so that (\ref{label_1}) with initial data $(\gamma_0^{(k)})$ has a unique solution in $L_t^\infty([0,T];H^1_{\xi'})\cap L_t^1([0,T];H^1_{\xi'})$.  Moreover, the solution satisfies conditions (A)--(D) as well.
\end{proposition}

The proof of Proposition \ref{label_prop1} is analogous to the argument given in \cite{CP_Cauchy} (see also \cite{CPHigher}), and is based on a fixed point argument proceeding from the integral form of the equation.  The preservation of the conditions (A)--(D) also follows from this argument and carries over identically; note that in the case of the positivity property (A) this is a delicate matter, and relies on the quantum de Finetti and uniqueness results for the relevant hierarchy (as we noted earlier, see \cite{ChTa} for the case of the GP hierarchy).

We next recall several implications of the admissibility conditions (A)--(D), which express symmetry properties of the initial data (and, in view of the uniqueness claim above, of solutions).  We begin with an identity from \cite{CPT} (see, e.g. (4.10)--(4.11) in \cite{CPT}): fix $d\geq 1$ and let $A:(x_1,x'_1)\mapsto A(x_1,x'_1)\in\mathcal{S}(\mathbb{R}^d\times\mathbb{R}^d)$ be given; then
\begin{align}
\label{label_11}\Tr(\Delta_{x_1}A)=\Tr(\Delta_{x'_1}A)=\Tr(-\nabla_{x_1}\cdot \nabla_{x'_1}A).
\end{align}

We also recall, as a basic consequence of the Hermitian property of (A)--(D)-admissible solutions, for $1\leq i\leq d$ we have the identities 
\begin{align}
(\partial_{x_i}\gamma^{(1)})(x,x)=\overline{(\partial_{x'_i}\gamma^{(1)})(x,x)},\label{label_35}
\end{align}
\begin{align}
(\partial_{x_i}\partial_{x_j}\gamma^{(1)})(x,x)=\overline{(\partial_{x'_i}\partial_{x'_j}\gamma^{(1)})(x,x)},\label{label_36}
\end{align}
and
\begin{align}
(\partial_{x_i}\partial_{x'_j}\gamma^{(1)})(x,x)=\overline{(\partial_{x'_i}\partial_{x_j}\gamma^{(1)})(x,x)},\label{label_37}
\end{align}
for $x\in\mathbb{R}^d$.  As a consequence, the identities 
\begin{align}
(\nabla_{x_1}\gamma^{(1)})(x,x)=\nabla_{x_1}[\overline{\gamma^{(1)}(x'_1,x_1)}]_{x_1=x'_1=x}=\overline{(\nabla_{x'_1}\gamma^{(1)})(x,x)}
\end{align}
and
\begin{align}
(\Delta_{x_1}\gamma^{(1)})(x,x)=\Delta_{x_1}[\overline{\gamma^{(1)}(x'_1,x_1)}]|_{x_1=x_1'=x}=\overline{(\Delta_{x'_1}\gamma^{(1)})(x,x)}\label{label_a1}
\end{align}
also hold.

In this context, we several times shall use the observation that $(\gamma^{(k)})\succeq 0$ and $(\gamma^{(k)})$ Hermitian together imply
\begin{align}
\gamma^{(1)}(x,x)\geq 0\quad\textrm{and}\quad \gamma^{(2)}(x,y,x,y)\geq 0\label{label_50}
\end{align}
for all $x,y\in\mathbb{R}^d$.

Associated to these observations, in the proof of Theorem $1.3$, we will make essential use of the quantum de Finetti theorem of Hudson and Moody \cite{HM}, as formulated in \cite{LNR} and \cite{CHPS}.  This result was used by Chen, Hainzl, Pavlovi\'c, and Seiringer in \cite{CHPS} and \cite{CHPS2} as an essential tool for studying uniqueness of solutions to the GP hierarchy.  Applying the theorem (e.g. as stated in \cite[Theorem $2.1$]{CHPS}) with the base Hilbert space taken as $H=L^2_{\rad}\subset L^2(\mathbb{R}^{d})$, the subspace of $L^2$ consisting of radial functions, this gives the existence of a Borel measure $\mu$ on $L_{\rad}^2$, supported on $\{f\in L_{\rad}^2:\lVert f\rVert_{L^2}=1\}$, with $\mu(L_{\rad}^2)=1$, and such that both
\begin{align}
\gamma^{(1)}(x,x)=\int |\phi(x)|^2d\mu(\phi)\label{label_48}
\end{align}
and
\begin{align}
\gamma^{(2)}(x,y,x,y)=\int |\phi(x)|^2|\phi(y)|^2d\mu(\phi)\label{label_49}
\end{align}
hold in the sense of distributions.

\section{A momentum identity and mass conservation}
\label{label_10}

In this section, we collect two preliminary conservation properties of (\ref{label_1}).  The first is the identification of a quantity
analogous to the momentum in a nonlinear Schr\"odinger equation and which obeys a pointwise conservation law, at least when the
solution is smooth.  A similar expression was identified for the GP hierarchy in \cite{CPT}.

\begin{lemma}
\label{label_26}
Let $(\gamma^{(k)})_{k\geq 1}$ be a smooth (A)--(D)-admissible solution to (\ref{label_1}) and define $P:\mathbb{R}\times\mathbb{R}^d\rightarrow\mathbb{C}^d$ by
\begin{align}
P(t,x)&:=\int e^{ix\cdot(\xi-\xi')}(\xi+\xi')\widehat{\gamma^{(1)}}(\xi,\xi')d\xi d\xi'\label{label_27}
\end{align}
for $t\in\mathbb{R}$ and $x\in\mathbb{R}^d$.  We then have
\begin{align*}
\partial_t \gamma^{(1)}(t,x,x)+\divop_x P(t,x)=0
\end{align*}
for all $(t,x)\in\mathbb{R}\times\mathbb{R}^d$.
\end{lemma}

\begin{proof}
Proceeding by direct computation, note that after taking the Fourier transform and its inverse (and using that $(\gamma^{(k)})$ solves ($\ref{label_1}$)) we obtain
\begin{align}
\nonumber \partial_t \gamma^{(1)}(x,x)&=\int e^{ix\cdot (\xi-\xi')}\partial_t \widehat{\gamma^{(1)}}(\xi,\xi')d\xi d\xi'\\
&=i\int e^{ix\cdot (\xi-\xi')}(|\xi'|^2-|\xi|^2)\widehat{\gamma^{(1)}}(\xi,\xi')d\xi d\xi'-\mu i(B^\pm_{1,2}\gamma^{(2)})(x,x).\label{label_28}
\end{align}

Now, writing $\partial_{x_i}e^{ix\cdot(\xi-\xi')}=i(\xi_i-\xi_i')e^{ix\cdot(\xi-\xi')}$ for $1\leq i\leq d$, we get
\begin{align*}
(\ref{label_28})&=-\divop_x\int e^{ix\cdot (\xi-\xi')}(\xi'+\xi)\widehat{\gamma^{(1)}}(\xi,\xi')d\xi d\xi'\\
&\hspace{0.2in}-\mu i\int \gamma^{(2)}(x,y,x',y)(V(x-y)-V(x'-y))\bigg|_{x'=x}dy\\
&=-\divop_x P,
\end{align*}
as desired.
\end{proof}

Using this result, we next establish conservation of mass for the hierarchy, that is, we show that the quantity $\Tr(\gamma^{(1)})$ remains invariant under the evolution.

\begin{proposition}[Conservation of mass for (\ref{label_1})]
\label{label_12}
Suppose that $(\gamma^{(k)})_{k\geq 1}$ is an (A)--(D)-admissible solution to the Hartree hierarchy (\ref{label_1}).  Then we have
\begin{align*}
\partial_t\Tr(\gamma^{(1)}(t))=0.
\end{align*}
\end{proposition}

\begin{proof}
As a consequence of the local theory stated above and standard approximation arguments, it suffices to show the result under the assumption that $\gamma^{(k)}\in C_t(I;\mathcal{S}(\mathbb{R}^{dk}\times\mathbb{R}^{dk}))$ for all $k\geq 1$.  By $(\ref{label_1})$, we get
\begin{align*}
\partial_t \Tr(\gamma^{(1)})&=\int \partial_t\gamma^{(1)}(x,x)dx\\
&=i\int (\Delta_{x_1}\gamma^{(1)})(x,x)-(\Delta_{x'_1}\gamma^{(1)})(x,x)dx-\mu i\int B_{1,2}^{\pm}\gamma^{(2)}(x,x)dx.
\end{align*}
Now, invoking ($\ref{label_a1}$) and ($\ref{label_36}$) to see that the first integral vanishes, and writing out the definition of the operator $B_{1,2}^\pm$, the right side of the above equality is equal to
\begin{align*}
-\mu i\int \gamma^{(2)}(x,y,x,y)\Big[V(x-y)-V(x'-y)\Big]\bigg|_{(x,x')=(x,x)}dydx=0,
\end{align*}
as desired.
\end{proof}

\section{Conservation of Energy}

We now establish a conservation law for the energy functional $E(t)$.

\begin{proposition}[Conservation of energy for (\ref{label_1})]
\label{label_13}
Suppose that $(\gamma^{(k)})_{k\geq 1}$ is an (A)--(D)-admissible solution to the Hartree hierarchy (\ref{label_1}).  Then 
\begin{align*}
\partial_t E(t)=0,
\end{align*}
where $E(t)$ denotes the quantity defined in (\ref{label_4}).
\end{proposition}

\begin{proof}
As in the proof of Proposition $2.3$, by standard approximation arguments it suffices to show the result when each $\gamma^{(k)}$ belongs to the class $C_t(I;\mathcal{S}(\mathbb{R}^{dk}\times\mathbb{R}^{dk}))$.  Recalling the definition of $E(t)$, and using that the sequence $(\gamma^{(k)})$ is a solution to the hierarchy $(\ref{label_1})$, we have
\begin{align}
\nonumber &\partial_tE(t)=\frac{i}{2}\Tr\bigg(-\Delta_{x_1}\bigg[\Delta_{x_1}\gamma^{(1)}-\Delta_{x'_1}\gamma^{(1)}-\mu B^{\pm}_{1,2}\gamma^{(2)}\bigg]\bigg)\\
\label{label_14}&\hspace{0.4in}+\frac{\mu i}{4}\Tr\bigg(B^+_{1,2}\bigg[\Delta^{(2)}_{\pm}\gamma^{(2)}-\mu \sum_{j=1}^2 B^{\pm}_{j,3}\gamma^{(3)}\bigg]\bigg).
\end{align}

We now observe that, taking the Fourier transform and its inverse, approximating $e^{ix_1\cdot(\xi_1-\xi'_1)}$ by $e^{-\delta|x_1|^2+ix_1\cdot(\xi_1-\xi'_1)}$ for $\delta>0$, applying Fubini's theorem, and taking limits as $\delta\rightarrow 0$ (by dominated convergence), one has the identity
\begin{align*}
&\Tr\bigg(-\Delta_{x_1}\bigg[\Delta_{x_1}\gamma^{(1)}(t)-\Delta_{x'_1}\gamma^{(1)}(t)\bigg]\bigg)\\
&\hspace{0.6in}=\lim_{\delta\rightarrow 0}\int e^{-\delta|x_1|^2+ix_1\cdot (\xi_1-\xi'_1)}|\xi_1|^2(|\xi'_1|^2-|\xi_1|^2)\widehat{\gamma^{(1)}}(\xi_1,\xi'_1)d\xi_1d\xi'_1dx_1\\
&\hspace{0.6in}=\int |\xi_1|^2(|\xi_1|^2-|\xi_1|^2)\widehat{\gamma^{(1)}}(\xi_1,\xi_1)d\xi_1=0,
\end{align*}
while expansion of the definitions of the operators $B_{1,2}^+$ and $B_{j,3}^\pm$ gives
\begin{align*}
&\Tr\bigg(B^+_{1,2}\bigg[\sum_{j=1}^2 B^{\pm}_{j,3}\gamma^{(3)}\bigg]\bigg)\\
&\hspace{0.2in}=\int\int \bigg[\sum_{j=1}^2 \int \gamma^{(3)}(x_1,x_2,z)(V(x_j-z)-V(x_j-z))dz\bigg]V(x_1-x_2)dx_2dx_1\\
&\hspace{0.2in}=0
\end{align*}
as well.

We therefore obtain that $\partial_tE(t)$ is equal to
\begin{align}
\nonumber &\frac{\mu i}{2}\Tr\bigg(\Delta_{x_1}\bigg[\int \gamma^{(2)}(x_1,y,x'_1,y)(V(x_1-y)-V(x'_1-y))dy\bigg]\bigg)\\
\nonumber &\hspace{0.2in}+\frac{\mu i}{4}\Tr\bigg(\int [(\Delta_{x_1}\gamma^{(2)})(x_1,y,x_1,y)-(\Delta_{x'_1}\gamma^{(2)})(x_1,y,x'_1,y)]V(x_1-y)dy\bigg)\\
&\hspace{0.2in}+\frac{\mu i}{4}\Tr\bigg(\int [(\Delta_{x_2}\gamma^{(2)})(x_1,y,x_1,y)-(\Delta_{x'_2}\gamma^{(2)})(x_1,y,x'_1,y)]V(x_1-y)dy\bigg).\label{label_15}
\end{align}

Now, using ($\ref{label_11}$),
\begin{align*}
&\frac{\mu i}{2}\Tr\bigg(\Delta_{x_1}\bigg[\int \gamma^{(2)}(x_1,y,x'_1,y)(V(x_1-y)\bigg]dy\bigg)\\
&\hspace{0.2in}=\frac{\mu i}{2}\Tr\bigg(\Delta_{x'_1}\bigg[\int \gamma^{(2)}(x_1,y,x'_1,y)(V(x_1-y)\bigg]dy\bigg)\\
&\hspace{0.2in}=\frac{\mu i}{2}\Tr\bigg(\int (\Delta_{x'_1}\gamma^{(2)})(x_1,y,x'_1,y)V(x_1-y)dy\bigg)
\end{align*}
while also
\begin{align*}
&\frac{\mu i}{2}\Tr\bigg(\Delta_{x_1}\bigg[\int \gamma^{(2)}(x_1,y,x'_1,y)(V(x'_1-y)\bigg]dy\bigg)\\
&\hspace{0.2in}=\frac{\mu i}{2}\Tr\bigg(\int (\Delta_{x_1}\gamma^{(2)})(x_1,y,x'_1,y)V(x'_1-y)dy\bigg).
\end{align*}
Combining this with the observation that symmetry of $\gamma^{(2)}$ with respect to permutations of the variables allows one to simplify the last two lines of (\ref{label_15}), we get
\begin{align*}
(\ref{label_15})&=\frac{\mu i}{2}\Tr\bigg(\int (\Delta_{x'_1}\gamma^{(2)})(x_1,y,x'_1,y)V(x_1-y)dy\bigg)\\
&\hspace{0.2in}-\frac{\mu i}{2}\Tr\bigg(\int (\Delta_{x_1}\gamma^{(2)})(x_1,y,x'_1,y)V(x'_1-y)dy\bigg)\\
&\hspace{0.2in}+\frac{\mu i}{2}\Tr\bigg(\int (\Delta_{x_1}\gamma^{(2)})(x_1,y,x'_1,y)V(x_1-y)dy\bigg)\\
&\hspace{0.2in}-\frac{\mu i}{2}\Tr\bigg(\int (\Delta_{x'_1}\gamma^{(2)})(x_1,y,x'_1,y)V(x_1-y)dy\bigg)\\
&=0.
\end{align*}
This completes the proof of the proposition.
\end{proof}

\begin{remark}
By the symmetry of $\gamma^{(k)}$, $k\geq 1$ with respect to permutations of the variables, Proposition $\ref{label_13}$ immediately implies that the quantities 
\begin{align*}
E_k(t)&:=\frac{1}{2}\Tr\bigg(\sum_{j=1}^k -\Delta_{x_j}\gamma^{(k)}(t)\bigg)+\frac{\mu}{4}\Tr\bigg(\sum_{j=1}^k B_{j,k+1}^+\gamma^{(k+1)}(t)\bigg)\\
&=k\bigg(\frac{1}{2}\Tr(\gamma^{(1)}(t))+\frac{\mu}{4}\Tr(B^+_{1,2}\gamma^{(2)}(t))\bigg),\quad k\geq 2,
\end{align*}
are also conserved.  A similar family of conserved quantities was observed in \cite{CPT} for the Gross-Pitaevskii hierarchy.
\end{remark}

\section{Virial and localized virial identities}

We next establish suitable forms of virial identities (and their localizations) for the hierarchy (\ref{label_1}).  These will ultimately be a key tool used to prove the finite-time blow-up results expressed in Theorem \ref{label_5} and Theorem \ref{label_6}.  The general procedure for the derivations follows a classical approach in analogy to similar results for the nonlinear Schr\"odinger equation; we note that a particularly relevant case of the analysis was performed in \cite{CPT} for the related GP hierarchy.

We begin with the relevant virial identity.

\begin{proposition}[Virial identity for (\ref{label_1})]
\label{label_29}
Let $(\gamma^{(k)})_{k\geq 1}$ be an (A)--(D)-admissible solution to (\ref{label_1}).  Then we have the identity
\begin{align*}
\partial_{tt}\Tr(|x|^2\gamma^{(1)})=8\Tr(-\Delta_{x_1}\gamma^{(1)})-4\mu \int \gamma^{(2)}(x,y,x,y)x\cdot (\nabla V)(x-y)dxdy.
\end{align*}
\end{proposition}

\begin{proof}
As we remarked in our arguments in the previous section, it suffices to show the result under the assumption that $\gamma^{(k)}$ is in $C_t([0,T];\mathcal{S}(\mathbb{R}^{dk}\times\mathbb{R}^{dk}))$ for all $k\geq 1$, as a consequence of the local well-posedness theory.  Invoking Lemma $\ref{label_26}$ and integrating by parts, we write
\begin{align}
\nonumber \partial_{tt}\Tr(|x|^2\gamma^{(1)})&=\int |x|^2\partial_{tt}\gamma(x,x)dx\\
\nonumber &=-\int |x|^2\divop_x\partial_t Pdx\\
\nonumber &=2\int x\cdot \partial_tPdx.
\end{align}
Next, expanding the definition of $P(t,x)$ and using that $(\gamma^{(k)})_{k\geq 1}$ solve ($\ref{label_1}$), this last expression becomes
\begin{align}
\nonumber &2\int e^{ix\cdot (\xi-\xi')}x\cdot(\xi+\xi')\partial_t \widehat{\gamma^{(1)}}(\xi,\xi')d\xi d\xi' dx \\
\nonumber &\hspace{0.2in}=2i\int e^{ix\cdot (\xi-\xi')}x\cdot (\xi+\xi')(|\xi'|^2-|\xi|^2)\widehat{\gamma^{(1)}}(\xi,\xi')d\xi d\xi' dx\\
&\hspace{0.4in}-2\mu i\int e^{ix\cdot (\xi-\xi')}x\cdot (\xi+\xi')\widehat{B^\pm_{1,2}\gamma^{(2)}}(\xi,\xi')d\xi d\xi' dx.\label{label_30}
\end{align}

To study the first term, we use an argument from \cite{CPT} (see in particular Section $5.2$ of \cite{CPT}).  This computation proceeds as follows: first, letting $x\otimes y$ denote the operator defined by $(x\otimes y)z=(y\cdot z)x$ for $x,y,z\in\mathbb{R}^d$, integration by parts and Fubini's theorem yield
\begin{align}
\nonumber &2i\int e^{ix\cdot (\xi-\xi')}x\cdot (\xi+\xi')(|\xi'|^2-|\xi|^2)\widehat{\gamma^{(1)}}(\xi,\xi')d\xi d\xi' dx\\
\nonumber &\hspace{0.2in}=-2i\int e^{ix\cdot (\xi-\xi')}x\cdot [(\xi+\xi')\otimes (\xi+\xi')](\xi-\xi')\widehat{\gamma^{(1)}}(\xi,\xi')d\xi d\xi' dx\\
\nonumber &\hspace{0.2in}=2\int e^{ix\cdot(\xi-\xi')}\tr[(\xi+\xi')\otimes (\xi+\xi')]\widehat{\gamma^{(1)}}(\xi,\xi')d\xi d\xi'dx\\
\label{label_31}&\hspace{0.2in}=8\int |\xi|^2\widehat{\gamma^{(1)}}(\xi,\xi)d\xi
\end{align}
where $\tr(A)=\sum_{i=1}^d (Ae_i)\cdot e_i$ is the usual matrix trace operator.  Now, using the assumption that $(\gamma^{(k)})_{k\geq 1}$ is Hermitian (since it is (A)--(D)-admissible), we may appeal to Parseval's identity to obtain
\begin{align*}
(\ref{label_31})&
=8\sum_{j=1}^\infty \int \lambda_j^2|(i\xi)\widehat{g}_j(\xi)|^2d\xi
=8\sum_{j=1}^\infty \int \lambda_j^2|\nabla g_j(x)|^2dx
=8\Tr(-\Delta_{x_1}\gamma^{(1)}),
\end{align*}
where $(\lambda_j)$ and $(g_j)$ are suitably chosen sequences.

Substituting these identities back into ($\ref{label_30}$), we obtain
\begin{align*}
(\ref{label_30})&=8\Tr(-\Delta_{x_1}\gamma^{(1)})+(II),
\end{align*}
where
\begin{align*}
(II):=-2\mu i\int e^{ix\cdot (\xi-\xi')}x\cdot (\xi+\xi')\widehat{B^\pm_{1,2}\gamma^{(2)}}(\xi,\xi')d\xi d\xi' dx.
\end{align*}

To complete the proof of the proposition, it remains to show 
\begin{align}
(II)=-4\mu \int \gamma^{(2)}(x,y,x,y)x\cdot (\nabla V)(x-y)dxdy.\label{label_32}
\end{align}

To accomplish this, we note that a direct calculation allows us to compute the Fourier transform of $B^+_{1,2}$ acting on $\gamma^{(2)}$ as
\begin{align*}
\widehat{B^{\pm}_{1,2}\gamma^{(2)}}(\xi,\xi')&=\int [\widehat{\gamma^{(2)}}(\xi-q+q',q,\xi',q')\\
&\hspace{1.2in}-\widehat{\gamma^{(2)}}(\xi,q,\xi'+q-q',q')]\widehat{V}(q-q')dq dq'.
\end{align*}
We therefore obtain
\begin{align*}
(II)&=-2\mu i\int e^{ix\cdot (\xi-\xi')}x\cdot (\xi+\xi') [\widehat{\gamma^{(2)}}(\xi-q+q',q,\xi',q')\\
&\hspace{1.2in}-\widehat{\gamma^{(2)}}(\xi,q,\xi'+q-q',q')]\widehat{V}(q-q')dqdq' d\xi d\xi' dx
\end{align*}
which, in view of the changes of variables $\xi\mapsto \xi-q+q'$ and $\xi'\mapsto \xi'-q+q'$ in the first and second terms, respectively, is equal to
\begin{align}
-2\mu i\int e^{ix\cdot (\xi-\xi')}x\cdot (2q-2q')\widehat{\gamma^{(2)}}(\xi-q+q',q,\xi',q')\widehat{V}(q-q') dq dq' d\xi d\xi' dx.\label{label_33}
\end{align}

Expanding $\widehat{\gamma^{(2)}}$ and $\widehat{V}$ in this expression by the definition of the Fourier transform, we obtain
\begin{align*}
(\ref{label_33})&=-4\mu i\int e^{i\Phi}x\cdot (q-q')\gamma^{(2)}(z,y,z',y')V(w) d\Lambda,
\end{align*}
where $\Lambda=(x,w,y,z,q,\xi,y',z',q',\xi')$ and with $\Phi=\Phi(\Lambda)$ defined by
\begin{align*}
\Phi(\Lambda):=(x-z)\cdot \xi-(x-z')\cdot \xi'+(z-y-w)\cdot q-(z-y'-w)\cdot q'.
\end{align*}

By the Fubini theorem, an approximation and limiting argument as before (first integrating in the variables $\xi$ and $z$, resulting in restriction to the set where $z=x$, and then integrating in the variables $\xi'$ and $z'$, resulting in restriction to the set where $z'=x$) shows that this is equal to 
\begin{align*}
&-4\mu \int x\cdot \nabla_x[e^{i\widetilde{\Phi}}]\gamma^{(2)}(x,y,x,y')V(w)d\widetilde{\Lambda}
\end{align*}
with $\widetilde{\Lambda}=(x,w,y,q,y',q')$ and $\widetilde{\Phi}=\widetilde{\Phi}(\widetilde{\Lambda})=(x-y-w)\cdot q-(x-y'-w)\cdot q'$.  Integrating by parts, this becomes
\begin{align*}
&4\mu d\int e^{i\widetilde{\Phi}}\gamma^{(2)}(x,y,x,y')V(w)d\widetilde{\Lambda}\\
&\hspace{0.2in}+4\mu\int e^{i\widetilde{\Phi}}x\cdot \nabla_x[\gamma^{(2)}(x,y,x,y'))]V(w)d\widetilde{\Lambda}
\end{align*}

Evaluating the integrals in $q'$, $y'$, $q$ and $w$, and recalling that $(\gamma^{(k)})$ is (A)--(D)-admissible (and thus in particular satisfies (D)), the above expression becomes
\begin{align*}
4\mu d \int B^+_{1,2}\gamma^{(2)}(x,x)dx+4\mu\int x\cdot \nabla_x[\gamma^{(2)}(x,y,x,y)]V(x-y)dxdy
\end{align*}
which, after another application of integration by parts, is equal to the desired quantity in ($\ref{label_32}$).
This completes the proof of the proposition.
\end{proof}

We next establish a localized version of Proposition $\ref{label_29}$, in which the weight $|x|^2$ is replaced with an arbitrary smooth cutoff function.

\begin{proposition}[Localized virial identity for (\ref{label_1})]
\label{label_34}
Fix $\phi\in C_c^\infty(\mathbb{R}^d)$ and let $(\gamma^{(k)})_{k\geq 1}$ be an (A)--(D)-admissible solution to (\ref{label_1}).  Then we have the identity
\begin{align*}
\partial_{tt}\Tr(\phi\gamma^{(1)})&=2\textrm{Re}\,\int H_{x}(\phi)(x)\cdot \bigg(H_{x,x'}(\gamma^{(1)})(x,x)-H_{x,x}(\gamma^{(1)})(x,x)\bigg)dx\\
&\hspace{0.4in}-2\mu \int \gamma^{(2)}(x,y,x,y)(\nabla \phi)(x)\cdot (\nabla V)(x-y)dxdy
\end{align*}
where $H_x(\phi)$, $H_{x,x}(f)$ and $H_{x,x'}(f)$ (with $f:I\times\mathbb{R}^{dk}\times \mathbb{R}^{dk}\rightarrow \mathbb{C}$) are the $d\times d$ matrices $H_x(\phi)=(\partial_{x_i}\partial_{x_j}\phi)_{i,j}$, $H_{x,x}(f)=(\partial_{x_i}\partial_{x_j}f)_{i,j}$ and $H_{x,x'}(f)=(\partial_{x_i}\partial_{x_j'}f)_{i,j}$.
\end{proposition}

\begin{proof}[Proof of Proposition $\ref{label_34}$]
As in the proof of Proposition $\ref{label_29}$, we may assume $\gamma^{(k)}\in C_t(I;\mathcal{S}(\mathbb{R}^{dk}\times\mathbb{R}^{dk}))$ for all $k\geq 1$, and use integration by parts to write
\begin{align}
\partial_{tt}\Tr(\phi\gamma^{(1)})&=\int \nabla_x \phi(x)\cdot \partial_t Pdx.\label{label_38}
\end{align}
where $P$ is defined as in Proposition $2.1$.  We then use the definition of $P$ to obtain
\begin{align}
\nonumber (\ref{label_38})&=\int e^{ix\cdot (\xi-\xi')}\nabla \phi(x) \cdot (\xi+\xi')\partial_t\widehat{\gamma^{(1)}}(\xi,\xi')d\xi d\xi'dx\\
\nonumber &=\int ie^{ix\cdot (\xi-\xi')}\nabla \phi(x)\cdot (\xi+\xi')(|\xi'|^2-|\xi|^2)\widehat{\gamma^{(1)}}(\xi,\xi')d\xi d\xi' dx\\
&\hspace{0.4in}-\mu i\int e^{ix\cdot (\xi-\xi')}\nabla \phi(x)\cdot (\xi+\xi')\widehat{{B^\pm_{1,2}\gamma^{(2)}}}(\xi,\xi')d\xi d\xi' dx.\label{label_39}
\end{align}

Performing a similar calculation as before on the first term, we obtain
\begin{align*}
&\int ie^{ix\cdot (\xi-\xi')}\nabla\phi(x)\cdot (\xi+\xi')(|\xi'|^2-|\xi|^2)\widehat{\gamma^{(1)}}(\xi,\xi')d\xi d\xi' dx\\
&\hspace{0.2in}=\sum_{j,k=1}^d \int (\partial_{x_k}\partial_{x_j}\phi)(x)e^{ix\cdot(\xi-\xi')}(\xi+\xi')_k(\xi+\xi')_j\widehat{\gamma^{(1)}}(\xi,\xi')d\xi d\xi' dx\\
&\hspace{0.2in}=\sum_{j,k} \int (\partial_{x_k}\partial_{x_j}\phi)(x)\bigg(-(\partial_{x_k}\partial_{x_j}\gamma^{(1)})(x,x)+(\partial_{x_k}\partial_{x'_j}\gamma^{(1)})(x,x)\\
&\hspace{0.6in}+(\partial_{x'_k}\partial_{x_j}\gamma^{(1)})(x,x)-(\partial_{x'_k}\partial_{x'_j}\gamma^{(1)})(x,x)\bigg)dx.
\end{align*}

In view of this, applying ($\ref{label_36}$) and ($\ref{label_37}$) and defining $(II_{\phi})$ by
\begin{align*}
(II_\phi):=-\mu i\int e^{ix\cdot (\xi-\xi')}\nabla \phi(x)\cdot (\xi+\xi')\widehat{{B^\pm_{1,2}\gamma^{(2)}}}(\xi,\xi')d\xi d\xi' dx,
\end{align*}
we obtain that the right side of ($\ref{label_39}$) is equal to
\begin{align*}
2\textrm{Re}\,\int H_{x}(\phi)(x)\cdot \bigg(H_{x,x'}(\gamma^{(1)})(x,x)-H_{x,x}(\gamma^{(1)})(x,x)\bigg)dx+(II_\phi).
\end{align*}

It remains to evaluate $(II_\phi)$.  As in Proposition $\ref{label_29}$, this is accomplished by a simple distributional calculation: computing $\widehat{B^\pm_{1,2}}$, expanding $\widehat{\gamma^{(2)}}$ and $\widehat{V}$ via the definition of the Fourier transform, and applying the Fubini theorem.  Indeed, this procedure of calculation gives
\begin{align*}
(II_\phi)&=-2\mu \int \gamma^{(2)}(x,y,x,y)(\nabla \phi)(x)\cdot (\nabla V)(x-y)dxdy,
\end{align*}
exactly as in the second half of the proof of Proposition $\ref{label_29}$.  This completes the proof of the proposition.
\end{proof}

To conclude this section, we include a brief lemma showing how the first term on the right side of the identity in Proposition $\ref{label_34}$ can be evaluated further, using the Hermitian property of (A)--(D)-admissible solutions.

\begin{lemma}
\label{label_40}
Fix $\phi\in C_c^\infty$, and let $(\gamma^{(k)})_{k\geq 1}$ be an (A)--(D)-admissible solution to (\ref{label_1}).  Then
\begin{align*}
&2\textrm{Re}\,\int H_{x}(\phi)\cdot H_{x,x}(\gamma^{(1)})(x,x)dx\\
&\hspace{0.2in}=\int \Delta^2(\phi)(x)\gamma^{(1)}(x,x)dx-2\textrm{Re}\,\int H_{x}(\phi)\cdot H_{x',x}(\gamma^{(1)})(x,x)dx,
\end{align*}
where $H_x(\phi)$, $H_{x,x}(f)$ and $H_{x',x}(f)$ are as defined in the statement of Lemma $\ref{label_34}$.
\end{lemma}
\begin{proof}
We may again assume that $\gamma^{(k)}$ is smooth for all $k$.  Since $(\gamma^{(k)})$ is Hermitian and $\gamma^{(k)}\succeq 0$, we may find sequences $(\lambda_\ell)$ and $(g_\ell)$ such that
\begin{align*}
\gamma^{(1)}(x,x')=\sum_{\ell=1}^\infty \lambda_\ell g_\ell(x)\overline{g_\ell(x')}
\end{align*}
for all $(x,x')\in\mathbb{R}^{2d}$.  We then integrate by parts (writing $\partial_{j,k}=\partial_{x_j}\partial_{x_k}$ and $\partial_{j,j,k}=\partial^2_{x_j}\partial_{x_k}$) to obtain
\begin{align*}
&\int H_{x}(\phi)(x)\cdot H_{x,x}(\gamma^{(1)})(x,x)dx\\
&\hspace{0.2in}=\sum_{j,k,\ell=1}^d\lambda_\ell \int(\partial_{j,k}\phi)(x)(\partial_{j,k}g_\ell)(x)\overline{g_\ell(x)}dx\\
&\hspace{0.2in}=-\sum_{j,k,\ell} \lambda_\ell \int\bigg[(\partial_{j,j,k}\phi)(x)\overline{g_\ell(x)}+(\partial_{j,k}\phi)(x)\overline{(\partial_{j}g_\ell)(x)}\bigg](\partial_{x_k}g_\ell)(x)dx
\end{align*}
A second integration by parts shows that this is equal to
\begin{align}
\nonumber &\sum_{j,k,\ell=1}^d \lambda_\ell \int\bigg((\partial_{j,j,k,k}\phi)\overline{g_\ell(x)}+(\partial_{j,j,k}\phi)\overline{\partial_{k}g_\ell(x)}\bigg)g_\ell(x)dx\\
&\hspace{0.4in}-\sum_{j,k,\ell=1}^d \lambda_\ell\int(\partial_{j,k}\phi)\overline{(\partial_{j}g_\ell)(x)}(\partial_{k}g_\ell)(x)dx\label{label_41}
\end{align}
A final integration by parts now gives
\begin{align*}
(\ref{label_41})&=\int \Delta^2(\phi)(x)\gamma^{(1)}(x,x)dx-\sum_{j,k,\ell}\lambda_\ell\int (\partial_{j,k}\phi)(x)\partial_{j}[\overline{\partial_{k}g_\ell(x)}g_{\ell}(x)]dx\\
&\hspace{0.2in}-\sum_{j,k,\ell} \lambda_\ell\int(\partial_{j,k}\phi)(x)\overline{(\partial_{j}g_\ell)(x)}(\partial_{k}g_\ell)(x)\\
&=\int \Delta^2(\phi)(x)\gamma^{(1)}(x,x)dx-\overline{\int H_{x}(\phi)\cdot H_{x,x}(\gamma^{(1)})(x,x)dx}\\
&\hspace{0.2in}-2\textrm{Re}\,\sum_{j,k,\ell} \lambda_\ell\int(\partial_{j,k}\phi)(x)\overline{(\partial_{j}g_\ell)(x)}(\partial_{k}g_\ell)(x),
\end{align*}
which yields
\begin{align*}
2\textrm{Re}\,\int H_{x}(\phi)\cdot H_{x,x}(\gamma^{(1)})(x,x)dx&=\int \Delta^2(\phi)(x)\gamma^{(1)}(x,x)dx\\
&\hspace{0.2in}-2\textrm{Re}\,\int H_{x}(\phi)\cdot H_{x',x}(\gamma^{(1)})(x,x)dx
\end{align*}
as desired.
\end{proof}

\section{Proofs of the main theorems: negative energy blow-up solutions}

In this section, we give the proofs of Theorem \ref{label_5} and Theorem $\ref{label_6}$, our results on negative energy finite-time blowup for ($\ref{label_1}$).  We begin with the proof of Theorem $\ref{label_5}$, for which our arguments are in the spirit of the classical Glassey argument (see also \cite{CPT} for a similar application of this argument to negative-energy blowup for the Gross-Pitaevskii hierarchy).

\begin{proof}[Proof of Theorem $\ref{label_5}$.]
Define $V_1(t)=\Tr(|x|^2\gamma^{(1)}(t))$.  Since $(\gamma^{(k)})$ is Hermitian with $\gamma^{(k)}\succeq 0$, we may find $\lambda_j\geq 0$ and $\psi_j:\mathbb{R}^d\rightarrow\mathbb{C}$ such that
\begin{align*}
V_1(t)=\sum_{j=1}^\infty \int |x|^2\lambda_j|\psi(x)|^2dx\geq 0.
\end{align*}
Suppose for contradiction that the claim fails.  Then $(\gamma^{(k)})_{k\geq 1}$ is a global solution, and thus $V_1(t)$ is defined for all $t\in\mathbb{R}$ with $V(t)\geq 0$ everywhere.  Recall that Proposition $\ref{label_29}$ implies the bound
\begin{align*}
\partial_{tt}V_1(t)&=8\Tr(-\Delta_{x_1}\gamma^{(1)})+4\int \gamma^{(2)}(x,y,x,y)x\cdot (\nabla V)(x-y)dxdy\\
&=16E(t)+4\int \gamma^{(2)}(x,y,x,y)V(x-y)dydx\\
&\hspace{0.2in}+4\int \gamma^{(2)}(x,y,x,y)x\cdot (\nabla V)(x-y)dydx\\
&=16E(t)+4\int \gamma^{(2)}(x,y,x,y)V(x-y)dydx\\
&\hspace{0.2in}+2\int \gamma^{(2)}(x,y,x,y)(x-y)\cdot (\nabla V)(x-y)dydx\\
&=16E(t)+4[(V+\frac{1}{2}x\cdot (\nabla V))*\gamma^{(2)}(x,\cdot,x,\cdot)](x)\\
&\leq 16E(t)
\end{align*}
where we have used the identity
\begin{align}
\nonumber &\int \gamma^{(2)}(x,y,x,y)x\cdot (\nabla V)(x-y)dydx\\
&\hspace{0.2in}=-\int \gamma^{(2)}(x,y,x,y)y\cdot (\nabla V)(x-y)dydx,\label{label_42}
\end{align}
which is a consequence of the symmetry property of $(\gamma^{(k)})$ and the assumption that $V$ is even.

However, $t\mapsto E(t)$ is constant by Proposition $\ref{label_13}$, the conservation of energy.  In particular, we obtain $V_1(t)<0$ for $|t|$ is sufficiently large, which contradicts the positivity of $V_1$.  Thus, we conclude that $(\gamma^{(k)})$ cannot be globally defined.
\end{proof}

We now turn to the proof of Theorem $\ref{label_6}$.  For these arguments, we make use of the truncated virial identity established in Proposition $\ref{label_34}$.  We begin with a variant of this proposition which is adapted to an assumption of radial symmetry and to the particular rescaled cutoff we wish to use.

Fix $\rho\in C_c^2(\mathbb{R})$ such that $\rho(x)\geq 0$ for all $x\in\mathbb{R}$, $\supp(\rho)\subset (1,3)$ with $\rho>0$ on $(\frac{5}{4},\frac{11}{4})$, $\int \rho(x)dx=1$, $\rho'\geq 0$ on $(1,\frac{3}{2})$, and $\rho(x)=\rho(4-x)$ for all $x\in\mathbb{R}$.  Define $\psi:[0,\infty)\rightarrow \mathbb{R}$ by
\begin{align*}
\psi(x)=x-\int_0^x (x-y)\rho(y)dy
\end{align*}
for $x\geq 0$, and for each $R>0$ let $\psi_R$ be given by 
\begin{align*}
\psi_R(x)=R\psi(\frac{|x|^2}{R})\quad\textrm{for}\quad x\in\mathbb{R}^d.
\end{align*}

\begin{lemma}
\label{label_43}
Set $\mu=-1$.  Suppose that $(\gamma^{(k)})_{k\geq 1}$ is an (A)--(D)-admissible solution to (\ref{label_1}) which is radially symmetric in the $x$ and $x'$ variables respectively (in the sense given in the statement of Theorem \ref{label_6}).  We then have
\begin{align}
\nonumber &\partial_{tt}\Tr(\psi_R\gamma^{(1)})\\
\nonumber &\hspace{0.2in}\leq 16 E(0)-8\int(1-\psi'(\frac{|x|^2}{R})-2|x|^2R^{-1}\psi''(\frac{|x|^2}{R})) (\partial_{r,r'}\gamma^{(1)})(x,x)dx\\
\nonumber &\hspace{0.4in}+4\int (V(x-y)+\frac{1}{2}(x-y)\cdot (\nabla V)(x-y)) \gamma^{(2)}(x,y,x,y) dxdy\\
\nonumber &\hspace{0.4in}-2\int a(x,y)\cdot (\nabla V)(x-y)\gamma^{(2)}(x,y,x,y)dxdy\\
&\hspace{0.4in}-\int\Delta^2(\psi_R)(x)\gamma^{(1)}(x,x)dx\label{label_44}
\end{align}
for every $R>0$, where we have written $r=|x|$, $r'=|x'|$, and set
\begin{align}
a(x,y):=(x-y)-(\psi'(\frac{|x|^2}{R})x-\psi'(\frac{|y|^2}{R})y),\quad (x,y)\in\mathbb{R}^{d}\times\mathbb{R}^d.\label{label_45}
\end{align}
\end{lemma}

This will be complemented by a ``truncation lemma'' estimating $a(x,y)$.  For some related estimates used in the analysis of the Hartree equation, see \cite{H} and the references cited there.

\begin{lemma}
\label{label_18}
For $x,y\in \mathbb{R}^d$, let $a(x,y)$ be the expression defined in (\ref{label_45}).  For all $R\geq 0$, define $F_R:[0,\infty)\rightarrow \mathbb{R}$ by
\begin{align}
F_R(r):=\int_0^{r^2/R}\rho(s)ds\label{label_17}
\end{align}
for $r\geq 0$.

Then  each $R\geq 1$ there exists $C>0$ such that for every $x,y\in\mathbb{R}^d$ with $\max\{|x|,|y|\}\geq R^{1/2}$ and $|x-y|\leq R^{1/2}$, we have
\begin{align}
\nonumber &|a(x,y)|\\
&\hspace{1.2in}\lesssim \bigg(F_R(|x|)+\frac{|x|^2}{R}\rho\Big(\frac{|x|^2}{R}\Big)+F_R(|y|)+\frac{|y|^2}{R}\rho\Big(\frac{|y|^2}{R}\Big)\bigg) |x-y|.\label{label_19}
\end{align}
\end{lemma}

The proofs of Lemma $\ref{label_43}$ and Lemma $\ref{label_18}$ are given in Appendix \ref{label_app_localized}.

We now continue with the proof of Theorem $\ref{label_6}$.  The general pattern of argument is closely related to the approach of Ogawa and Tsutsumi \cite{OT}; see also \cite{H} for an earlier application of the method to the nonlinear Schr\"odinger
equation with Hartree nonlinearity, as well as a textbook treatment in \cite[Theorem $6.5.10$]{C}.  

\begin{proof}[Proof of Theorem $\ref{label_6}$.]
For each $R>0$, let $\rho$, $\psi$, and $F_R$ be as defined in Section $2.2$, and define $\psi_R$ by $\psi_R(x)=R\psi(\frac{|x|^2}{R})$ as in Lemma $\ref{label_43}$.  Applying Lemma $\ref{label_43}$ and using ($\ref{label_7}$) and ($\ref{label_50}$) to see that the third term of the resulting bound is non-positive, we bound $\partial_{tt}\Tr(\psi_R\gamma^{(1)})$ by 
\begin{align}
16 E(0)+(II)+(III)+(IV),\label{label_51}
\end{align}
with
\begin{align*}
(II):=-8\int\bigg(1-\psi'(\frac{|x|^2}{R})-2|x|^2R^{-1}\psi''(\frac{|x|^2}{R})\bigg) (\partial_{r,r'}\gamma^{(1)})(x,x)dx,
\end{align*}
\begin{align*}
(III):=-2\int a(x,y)\cdot (\nabla V)(x-y)\gamma^{(2)}(x,y,x,y)dxdy,
\end{align*}
and
\begin{align*}
(IV):=-\int\Delta^2(\psi_R)(x)\gamma^{(1)}(x,x)dx.
\end{align*}

As in the proof of Theorem $\ref{label_5}$, our goal is to bound this by a negative quantity, uniformly in $t$.  Since $E(0)<0$ by assumption, it suffices to estimate the sum of $(II)$, $(III)$, and $(IV)$.  

We begin by re-expressing $(II)$ in a more convenient form (recognizing it as a strictly negative quantity).  Note that
\begin{align*}
1-\psi'(\frac{|x|^2}{R})-2|x|^2R^{-1}\psi''(\frac{|x|^2}{R})=2|x|^2R^{-1}\rho\bigg(\frac{|x|^2}{R}\bigg)+\int_0^{|x|^2R^{-1}} \rho(y)dy
\end{align*}
for every $x\in\mathbb{R}^d$.  Combined with the integral expression ($\ref{label_48}$) for $\gamma^{(1)}$ (with respect to the measure $\mu$ on $L^2_{\rad}$), this leads (via an application of the Tonelli theorem to interchange the order of integration) to the representation
\begin{align}
\nonumber (II)&=-8\int \bigg(F_R(|x|)+\frac{2|x|^2}{R}\rho(\frac{|x|^2}{R})\bigg)(\partial_{r,r'}\gamma^{(1)})(x,x)dx\\
&=-8\int \bigg(F_R(|x|)+\frac{2|x|^2}{R}\rho(\frac{|x|^2}{R})\bigg)|\nabla \phi(x)|^2dxd\mu(\phi).\label{label_52}
\end{align}

We now turn to $(III)$.  Observing that $a(x,y)=0$ on $(\mathbb{R}^d\times\mathbb{R}^d)\setminus A_R$, with $A_R:=\{(x,y):\max\{|x|,|y|\}\geq R^{1/2}\}$, this term becomes
\begin{align}
\nonumber &-2\int_{A_R} a(x,y)\cdot (\nabla V)(x-y)\gamma^{(2)}(x,y,x,y)dxdy\\
\nonumber &\hspace{0.2in}\leq 2\int_{A_R\cap \{(x,y):|x-y|>R^{1/2}\}} |a(x,y)\cdot (\nabla V)(x-y)|\gamma^{(2)}(x,y,x,y)dxdy\\
&\hspace{0.4in}+2\int_{A_R\cap \{(x,y):|x-y|\leq R^{1/2}\}} |a(x,y)\cdot (\nabla V)(x-y)|\gamma^{(2)}(x,y,x,y)dxdy.\label{label_53}
\end{align}
where we have used ($\ref{label_50}$).  We let the first and second terms in ($\ref{label_53}$) be denoted by $(IIIa)$ and $(IIIb)$, respectively.

To estimate $(IIIa)$, note that we may find $C>0$ such that $|a(x,y)|\leq C|x-y|$ for all $x$ and $y$.  We then have
\begin{align*}
(IIIa)&\leq C\int_{|x-y|\geq R^{1/2}} |x-y|\,|(\nabla V)(x-y)|\gamma^{(2)}(x,y,x,y)dxdy\\
&\leq C\sup_{|x-y|\geq R^{1/2}} \bigg(|x-y|\,|(\nabla V)(x-y)|\bigg)\int \gamma^{(2)}(x,y,x,y)dxdy\\
&=C\sup_{|x-y|\geq R^{1/2}} \bigg(|x-y|\,|(\nabla V)(x-y)|\bigg)\int \gamma_0^{(1)}(x,x)dx.
\end{align*}
where we have used the admissibility of $\gamma^{(1)}$ and Proposition $\ref{label_12}$ to obtain the last equality.  We may then use the hypothesis $(\ref{label_8})$ to choose $R$ sufficiently large so that 
\begin{align}
(IIIa)\leq 4|E(0)|.\label{label_54}
\end{align}

We now turn to $(IIIb)$.  
Applying Lemma $\ref{label_18}$, we obtain
\begin{align*}
(IIIb)&\leq 2C\int_{A_R\cap \{(x,y):|x-y|\leq R^{1/2}\}} \bigg(F_R(|x|)+F_R(|y|)\\
&\hspace{2.2in}+\frac{|x|^2}{R}\rho\bigg(\frac{|x|^2}{R}\bigg)+\frac{|y|^2}{R}\rho\bigg(\frac{|y|^2}{R}\bigg)\bigg)|x-y|\\
&\hspace{2.4in} \cdot|(\nabla V)(x-y)|\gamma^{(2)}(x,y,x,y)dxdy,
\end{align*}
which, in view of the symmetry properties of (A)--(D)-admissible solutions, is bounded by a multiple of
\begin{align}
\nonumber &\int_{A_R\cap \{(x,y):|x-y|\leq R^{1/2}\}} \bigg(F_R(|x|)+\frac{|x|^2}{R}\rho\bigg(\frac{|x|^2}{R}\bigg)\bigg)|x-y|\\
&\hspace{1.8in}\cdot|(\nabla V)(x-y)|\gamma^{(2)}(x,y,x,y)dxdy.\label{label_55}
\end{align}

We now make use of ($\ref{label_49}$), the integral representation for $\gamma^{(2)}$ given by the quantum de Finetti theorem.  Substituting this into our bound for $(IIIb)$, we get (in view of the Tonelli theorem, and the H\"older and Young inequalities)
\begin{align*}
(\ref{label_55})&\lesssim \int \bigg(F_R(|x|)+\frac{|x|^2}{R}\rho\bigg(\frac{|x|^2}{R}\bigg)\bigg)|\phi(x)|^2\Big((|x|\,|\nabla V(x)|\chi_{R}(x)) \ast |\phi|^2\Big)(x)dxd\mu(\phi)\\
&\leq \int \bigg\lVert \bigg(F_R(|x|)+\frac{|x|^2}{R}\rho(\frac{|x|^2}{R}\bigg)\bigg)|\phi(x)|^2\bigg\rVert_{L_x^\infty}\\
&\hspace{1.2in} \cdot\lVert (|x|\,|\nabla V(x)|\chi_R(x))\ast |\phi(x)|^2\rVert_{L_x^1}d\mu(\phi)\\
&\lesssim \int \bigg\lVert \bigg(F_R(|x|)+\frac{|x|^2}{R}\rho\bigg(\frac{|x|^2}{R}\bigg)\bigg)|\phi(x)|^2\bigg\rVert_{L_x^\infty}\\
&\hspace{1.2in} \cdot \lVert |x|\,|\nabla V(x)|\rVert_{L_x^1(|x|\leq R^{1/2})}\lVert \phi\rVert_{L_x^2}^2d\mu(\phi),
\end{align*}
where we have let $\chi_{R}=\chi_{\{|x|\leq R^{1/2}\}}$ denote the characteristic function of the ball of radius $R^{1/2}$ centered at the origin.

Now, note that $F_R(|x|)=0$ and $\rho(|x|^2/R)=0$ for $|x|^2\leq R$, and, moreover, there exists $C>0$ such that $F_R(|x|)\leq C$ and $\frac{|x|^2}{R}\rho(|x|^2/R)\leq C$ for all $x\in\mathbb{R}^d$.  Combining these bounds with the equality 
\begin{align}
\lVert \phi\rVert_{L^2}=1,\label{label_56}
\end{align}
which is valid on the support of $\mu$, the above expression is bounded by a multiple of
\begin{align}
A_R\int \bigg\lVert \bigg(F_R(|x|)+\frac{|x|^2}{R}\rho\bigg(\frac{|x|^2}{R}\bigg)\bigg)^{1/2}\phi\bigg\rVert_{L_x^\infty(|x|\geq R^{1/2})}^2 d\mu(\phi),\label{label_57}
\end{align}
with $A_R:=\lVert |x|\, |\nabla V(x)\rVert_{L_x^1(|x|\leq R^{1/2})}$.

Recalling that $\mu$ is a measure on $L^2_{\rad}$, we now invoke a form of the Strauss lemma for radial functions (see, e.g. Lemma $1.7.3$ and Lemma $6.5.11$ in \cite{C}), giving the inequality
\begin{align*}
\lVert |x|^{(N-1)/2}f(x)g(x)\rVert_{L_x^\infty}^2&\lesssim \lVert f\nabla f\rVert_{L_x^\infty}\lVert g\rVert_{L_x^2}^2+\lVert fg\rVert_{L_x^2}\lVert f\nabla g\rVert_{L_x^2}\\
&\lesssim \lVert f\nabla f\rVert_{L_x^\infty}\lVert g\rVert_{L_x^2}^2+\lVert fg\rVert_{L_x^2}^2+\lVert f\nabla g\rVert_{L_x^2}^2
\end{align*}
with $f\in C^1(\mathbb{R}^N)$ and $g\in H^1(\mathbb{R}^N)$ both radial, where we have fixed $N\geq 2$.  Applying this inequality with $f(x)=(F_R(|x|)+\frac{|x|^2}{R}\rho(\frac{|x|^2}{R}))^{1/2}$, $x\in\mathbb{R}^d$ and $g=\phi$, we compute
\begin{align*}
|(f\nabla f)(x)|\lesssim \bigg(\frac{|x|}{R}+\frac{|x|^3}{R^2}\bigg)\lVert\rho\rVert_{C^1(\mathbb{R})}\chi_{\{x:1<|x|^2/R<3\}}(x)\lesssim R^{-1/2}
\end{align*}
for $x\in\mathbb{R}^d$, and therefore obtain
\begin{align}
\nonumber (\ref{label_57})&\lesssim \frac{A_R}{R^{(d-1)/2}}\int \bigg(R^{-1/2}\lVert \phi\rVert_{L_x^2}^2+\lVert \phi\rVert_{L^2}^2\\
\nonumber &\hspace{0.8in} +\bigg\lVert \bigg(F_R(|x|)+\frac{|x|^2}{R}\rho\bigg(\frac{|x|^2}{R}\bigg)\bigg)^{1/2}|\nabla \phi|\bigg\rVert_{L_x^2}^2\bigg)d\mu(\phi)\\
&\lesssim \frac{A_R}{R^{(d-1)/2}}\int \bigg(1+\bigg\lVert \bigg(F_R(|x|)+\frac{2|x|^2}{R}\rho\bigg(\frac{|x|^2}{R}\bigg)\bigg)^{1/2}|\nabla \phi|\bigg\rVert_{L_x^2}^2\bigg)d\mu(\phi),\label{label_58}
\end{align}
where we have used the observations made above about boundedness and support of $x\mapsto F_R(|x|)$ and $x\mapsto \frac{|x|^2}{R}\rho(|x|^2/R)$, and again invoked ($\ref{label_56}$).

Rewriting the right side of (\ref{label_58}) as
\begin{align*}
&\frac{A_R}{R^{(d-1)/2}}+\frac{A_R}{R^{(d-1)/2)}}\int \bigg(F_R(|x|)+\frac{2|x|^2}{R}\rho\bigg(\frac{|x|^2}{R}\bigg)\bigg)|\nabla \phi|^2dx d\mu(\phi)\\
&\hspace{0.2in}=\frac{A_R}{R^{(d-1)/2}}+\frac{A_R}{8R^{(d-1)/2}}|(II)|,
\end{align*}
where the last equality follows from (\ref{label_52}), the hypothesis (\ref{label_9}) implies that for $R$ sufficiently large (to ensure $CA_R/R^{(d-1)/2}\leq 4|E(0)|$ and $CA_R/(8R^{(d-1)/2})\leq 1/2$) we have the bound
\begin{align}
(IIIb)&\leq 4|E(0)|+\frac{1}{2}|(II)|.\label{label_59}
\end{align}

We now estimate $(IV)$.  Note that $\Delta^2\psi_R=0$ for $|x|\geq 3R^{1/2}$, and that
\begin{align*}
\Delta^2\psi_R&=\frac{16|x|^4}{R^3}\psi''''(\frac{|x|^2}{R})+\frac{16|x|^2(d+2)}{R^2}\psi'''(\frac{|x|^2}{R})+\frac{4d(d+2)}{R}\psi''(\frac{|x|^2}{R}).
\end{align*}
We then get the bound
\begin{align}
\nonumber (IV)&\leq \lVert \Delta^2\psi_R\rVert_{L^\infty(|x|\leq 2R^{1/2})}\int \gamma^{(1)}(t,x,x)dx\\
\label{label_60}&\lesssim R^{-1}\int \gamma^{(1)}(0,x,x) dx
\end{align}
where we have used Proposition $\ref{label_12}$ to obtain the last equality.  We may then choose $R$ sufficiently large so that 
\begin{align}
(IV)\leq 4|E(0)|.\label{label_61}
\end{align}  

Combining (\ref{label_51}) with (\ref{label_54}), (\ref{label_59}) and (\ref{label_61}), we obtain
\begin{align*}
\partial_{tt}\Tr(\psi_R\gamma^{(1)})&\leq 4E(0)+\frac{1}{2}(II)\leq 4E(0)
\end{align*}
for $R$ sufficiently large, where we have used that the identity (\ref{label_52}) implies $(II)\leq 0$.  Since this quantity is independent of $t$ and strictly negative, while $\Tr(\psi_R\gamma^{(1)})$ is strictly positive for all $t$ in the interval of existence, the result follows as before.
\end{proof}

\appendix

%\section{Local well posedness: proof of Proposition \ref{label_prop1}}
%
%\label{label_app_lwp}
%
%In this section we given the proof of Proposition $\ref{label_prop1}$, the local well-posedness result.  As we described earlier, the proof is related to similar results given in \cite{CP_Cauchy} and \cite{CPHigher} for the GP hierarchy, and is based on a fixed-point argument.
%
%\begin{proof}[Proof of Proposition \ref{label_prop1}]
%
%\end{proof}

\section{Localized virial identities adapted to scaling}
\label{label_app_localized}

In this brief appendix we give the proofs of Lemma $5.1$ and Lemma $5.2$.  These lemmas are used in the proof of Theorem $1.2$.  For convenience, we recall that $\rho\in C_c^2(\mathbb{R};[0,\infty))$ satisfies $\supp(\rho)\subset (1,3)$, $\rho>0$ on $(\frac{5}{4},\frac{11}{4})$, $\int \rho(x)dx=1$, $\rho'\geq 0$ on $(1,\frac{3}{2})$, and $\rho(x)=\rho(4-x)$ for all $x\in\mathbb{R}$, and $\psi$ and $\psi_R$, $R>0$, are given by
\begin{align*}
\psi(x)=x-\int_0^x (x-y)\rho(y)dy,\quad x\geq 0,
\end{align*}
and
\begin{align*}
\psi_R(x)=R\psi(\frac{|x|^2}{R})\quad\textrm{for}\quad x\in\mathbb{R}^d.
\end{align*}

We begin with the proof of Lemma $\ref{label_43}$.

\begin{proof}[Proof of Lemma \ref{label_43}]
Let $R>0$ be given.  As in our earlier results, there is no loss of generality (in view of the local theory described in Section $2$) in assuming $\gamma^{(k)}\in C_t(I;\mathcal{S}(\mathbb{R}^{dk}\times\mathbb{R}^{dk}))$ for all $k\geq 1$.  Applying Proposition $\ref{label_34}$ and Lemma $\ref{label_40}$, we obtain (since we have taken $\mu=-1$)
\begin{align}
\nonumber &\partial_{tt}\Tr(\psi_R\gamma^{(1)})\\
\nonumber &\hspace{0.4in}=4\textrm{Re}\,\int H_{x}(\psi_R)(x)\cdot H_{x,x'}(\gamma^{(1)})(x,x)dx-\int \Delta^2(\psi_R)(x)\gamma^{(1)}(x,x)dx\\
\nonumber &\hspace{0.8in}+2\int \gamma^{(2)}(x,y,x,y)(\nabla \psi_R)(x)\cdot (\nabla V)(x-y)dxdy.
\end{align}

Setting $r=|x|$, $r'=|x'|$, and computing derivatives of $\psi_R$, this is equal to
\begin{align}
\nonumber &4\textrm{Re}\,\sum_{j,k}\int \partial_{x_j}[2x_k\psi'(\frac{|x|^2}{R})](\partial_{x_j}[\partial_{r'}\gamma^{(1)}\frac{x'_k}{|x'|}])\bigg|_{(x,x)}-\int \Delta^2(\psi_R)(x)\gamma^{(1)}(x,x)dx\\
\nonumber &\hspace{0.4in}+4\int \psi'(\frac{|x|^2}{R})(x)\gamma^{(2)}(x,y,x,y)x\cdot (\nabla V)(x-y)dxdy.
\end{align}
Integrating by parts, this becomes
\begin{align}
\nonumber &8\textrm{Re}\,\sum_j \int \bigg[\frac{x_j^2}{|x|^2}\psi'(\frac{|x|^2}{R})(\partial_{rr'}\gamma^{(1)})(x,x)\\
\nonumber &\hspace{1.2in}+\sum_k\frac{2x^2_kx^2_j}{R|x|^2}\psi''(\frac{|x|^2}{R})(\partial_{rr'}\gamma^{(1)})(x,x)\bigg]dx\\
\nonumber &\hspace{0.2in}-\int\Delta^2(\psi_R)(x)\gamma^{(1)}(x,x)dx\\
&\hspace{0.2in}+4\int \psi'(\frac{|x|^2}{R})(x)\gamma^{(2)}(x,y,x,y)x\cdot (\nabla V)(x-y)dxdy.\label{label_46}
\end{align}

Now, evaluating the sum in $j$ and $k$, and observing that since $(\gamma^{(k)})$ is (A)--(D)-admissible (and in particular Hermitian), $\partial_{r,r'}\gamma^{(1)}(x,x)$ is real for all $x\in\mathbb{R}^d$, we obtain 
\begin{align}
\nonumber (\ref{label_46})&=8\int \bigg[\psi'(\frac{|x|^2}{R})(\partial_{rr'}\gamma^{(1)})(x,x)dx+\frac{2|x|^2}{R}\psi''(\frac{|x|^2}{R})(\partial_{rr'}\gamma^{(1)})(x,x)\bigg]dx\\
\nonumber &\hspace{0.2in}+4\int \psi'(\frac{|x|^2}{R})(x)\gamma^{(2)}(x,y,x,y)x\cdot (\nabla V)(x-y)dxdy\\
&\hspace{0.2in}-\int\Delta^2(\psi_R)(x)\gamma^{(1)}(x,x)dx,\label{label_47}
\end{align}
We next add and subtract $16E(0)$, where $E(t)$ is as the energy defined in ($\ref{label_4}$), which is conserved in time by Proposition $\ref{label_13}$.  This gives
\begin{align*}
(\ref{label_47})&=16 E(0)-8\int(1-\psi'(\frac{|x|^2}{R})-\frac{2|x|^2}{R}\psi''(\frac{|x|^2}{R})) (\partial_{r,r'}\gamma^{(1)})(x,x)dx\\
&\hspace{0.2in}+4\int \gamma^{(2)}(x,y,x,y)(V(x-y)+\frac{1}{2}(x-y)\cdot (\nabla V)(x-y))dxdy\\
&\hspace{0.2in}-2\int \gamma^{(2)}(x,y,x,y)(x-y)\cdot (\nabla V)(x-y)dxdy\\
&\hspace{0.2in}+4\int \psi'(\frac{|x|^2}{R})\gamma^{(2)}(x,y,x,y)x\cdot (\nabla V)(x-y)dxdy\\
&\hspace{0.2in}-\int\Delta^2(\psi_R)(x)\gamma^{(1)}(x,x)dx.
\end{align*}

The desired result now follows from the identity,
\begin{align*}
\nonumber &\int \psi'(\frac{|x|^2}{R})\gamma^{(2)}(x,y,x,y)x\cdot (\nabla V)(x-y)dydx\\
&\hspace{0.2in}=-\int \psi'(\frac{|y|^2}{R})\gamma^{(2)}(x,y,x,y)y\cdot (\nabla V)(x-y)dydx,
\end{align*}
which, in a similar manner to the related identity ($\ref{label_42}$) used in the proof of Theorem $\ref{label_5}$, follows from the symmetry properties of (A)--(D)-admissible solutions.
\end{proof}

We now turn to the proof of Lemma \ref{label_18}, which is an estimate for the quantity $a(x,y)$ defined in (\ref{label_45}).

\begin{proof}[Proof of Lemma \ref{label_18}]
Since $\psi'(t)=1-\int_0^{t}\rho(s)ds$ for $t\geq 0$, we may define 
\begin{align*}
f(t):=(x+t(y-x))F_R(|x+t(y-x)|)
\end{align*}
for $0\leq t\leq 1$, and observe that with this notation the left side of ($\ref{label_19}$) becomes
$|f(1)-f(0)|$.  To estimate this quantity, we write 
\begin{align}
|f(1)-f(0)|=\bigg|\int_0^1 f'(s)ds\bigg|\leq \int_0^1 |f'(s)|ds\label{label_20}
\end{align}
and estimate
\begin{align}
\nonumber |f'(s)|&\leq |y-x|\sup_s |F_R(|x+s(y-x)|)|\\
&\hspace{0.4in}+\sup_s |(x+s(y-x))\frac{d}{ds}[F_R(|x+s(y-x)|)]|.\label{label_21}
\end{align}

Now, since
$|x+s(y-x)|\leq \max\{|x|,|y|\}$ holds for $0\leq s\leq 1$, it follows from the observation that $F_R$ is increasing on $[0,\infty)$ that 
\begin{align}
F_R(|x+s(y-x)|)\leq \max\{F_R(|x|),F_R(|y|)\}\leq F_R(|x|)+F_R(|y|).\label{label_22}
\end{align}

On the other hand, explicit computation gives
\begin{align}
\frac{d}{ds}[F_R(|x+s(y-x)|)]&=\rho(|x+s(y-x)|^2/R)\frac{2(x+s(y-x))\cdot (y-x)}{R},\label{label_23}
\end{align}
so that (combining ($\ref{label_20}$), ($\ref{label_21}$), and ($\ref{label_22}$) with ($\ref{label_23}$)) we obtain
\begin{align}
\nonumber |f(1)-f(0)|&\leq \bigg(F_R(|x|)+F_R(|y|)\\
&\hspace{0.4in}+\sup_s \frac{2|x+s(y-x)|^2}{R}\rho(|x+s(y-x)|^2/R)\bigg)|x-y|.\label{label_24}
\end{align}

To estimate the supremum, We now consider three cases depending on the sizes of $|x|^2/R$ and $|y|^2/R$.  

\vspace{0.2in}

\noindent \underline{Case $1$}: Suppose first that $|x|^2/R<\frac{3}{2}$ and $|y|^2/R<\frac{3}{2}$.  Invoking once again $|x+s(y-x)|\leq \max\{|x|,|y|\}$ for $s\in [0,1]$, we then have
\begin{align*}
\frac{|x+s(y-x)|^2}{R}\leq \max\bigg\{\frac{|x|^2}{R},\frac{|y|^2}{R}\bigg\}\leq \frac{3}{2}
\end{align*}
for all such $s$.  Since $\rho$ is increasing on $(1,\frac{3}{2})$, this gives
\begin{align*}
\sup_s \frac{|x+s(y-x)|^2}{R}\rho\bigg(\frac{|x+s(y-x)|^2}{R}\bigg)&\leq \max\bigg\{\frac{|x|^2}{R}\rho\bigg(\frac{|x|^2}{R}\bigg),\frac{|y|^2}{R}\rho\bigg(\frac{|y|^2}{R}\bigg)\bigg\}\\
&\leq \frac{|x|^2}{R}\rho\bigg(\frac{|x|^2}{R}\bigg)+\frac{|y|^2}{R}\rho\bigg(\frac{|y|^2}{R}\bigg)
\end{align*}
which leads to the desired bound in this case.

\vspace{0.2in}

\noindent \underline{Case $2$}: We now consider the case when either (i) $|x|^2/R<\frac{3}{2}$ and $|y|^2/R\geq \frac{3}{2}$; or (ii) $|x|^2/R\geq \frac{3}{2}$ and $|y|^2/R<\frac{3}{2}$.  Since the estimate obtained is symmetric in $x$ and $y$, we may assume that we are in case (i) without any loss of generality.  In this setting, $|x+s(y-x)|\leq \max\{|x|,|y|\}=|y|$ for $s\in [0,1]$ gives
\begin{align*}
\sup_s \frac{|x+s(y-x)|^2}{R}\rho\bigg(\frac{|x+s(y-x)|^2}{R}\bigg)\leq \frac{|y|^2}{R}\lVert \rho\rVert_{L^\infty}.
\end{align*}
Noting that the triangle inequality and the hypothesis $|x-y|\leq R^{1/2}$ imply
\begin{align*}
\frac{|y|^2}{R}\leq \frac{(|y-x|+|x|)^2}{R}\leq (1+\sqrt{3/2})^2,
\end{align*}
and observing that $|y|^2/R\geq \frac{3}{2}$ implies
\begin{align}
F_R(|y|)\geq \int_0^{3/2} \rho ds=:c_0(\rho)>0,\label{label_25}
\end{align}
we get the bound
\begin{align*}
\sup_s \frac{|x+s(y-x)|^2}{R}\rho\bigg(\frac{|x+s(y-x)|^2}{R}\bigg)&\leq C\lVert \rho\rVert_{L^\infty}\leq \frac{C\lVert\rho\rVert_{L^\infty}}{c_0(\rho)}F_R(|y|)
\end{align*}
which again leads to the desired inequality.

\vspace{0.2in}

\noindent \underline{Case $3$}:  Suppose now that $|x|^2/R\geq \frac{3}{2}$ and $|y|^2/R\geq \frac{3}{2}$.  In this case we note that, in view of the assumption $|x-y|\leq R^{1/2}$, the condition $|x|^2/R\geq 10$ implies $\rho(|x+s(y-x)|^2/R)=0$ for all $0\leq s\leq 1$.  Indeed, if $|x|^2/R\geq 10$, we have
\begin{align*}
|x+s(y-x)|\geq |x|-|y-x|\geq (\sqrt{10}-1)R^{1/2}
\end{align*}
for all $s\in [0,1]$, so that
\begin{align*}
\inf_{s\in [0,1]}\frac{|x+s(y-x)|^2}{R}\geq (\sqrt{10}-1)^2>3,
\end{align*}
and the desired conclusion follows from the hypothesis $\supp \rho\subset (1,3)$.

It follows that whenever
\begin{align*}
\rho_*(x,y):=\sup_{s\in [0,1]} \rho\bigg(\frac{|x+s(y-x)|^2}{R}\bigg)
\end{align*}
is nonzero, we have
\begin{align*}
\frac{|x+s(y-x)|^2}{R}\leq \frac{2(|x|^2+|y-x|^2)}{R}\leq 22.
\end{align*}
Choosing $c_0(\rho)$ as in ($\ref{label_25}$), we therefore obtain
\begin{align*}
\sup_s \frac{|x+s(y-x)|^2}{R}\rho\bigg(\frac{|x+s(y-x)|^2}{R}\bigg)&\leq 22\lVert \rho\rVert_{L^\infty}\leq \frac{22}{c_0(\rho)}F_R(|y|).
\end{align*}
This again leads to the desired estimate.

\vspace{0.2in}

Since these three cases cover all possible values of $|x|$ and $|y|$, this completes the proof of the lemma.
\end{proof}

\end{document}